\documentclass{amsart}
\usepackage{graphicx} % Required for inserting images

\usepackage[utf8]{inputenc}
\usepackage[T1]{fontenc}

\usepackage{amssymb}
\usepackage{enumerate}
\usepackage{xcolor}
\definecolor{darkgreen}{rgb}{0,0.5,0}
\usepackage{tikz}
\usetikzlibrary{positioning,arrows,shapes,shadows}
\usepackage{caption}
\usepackage{subcaption}
\usepackage{cancel}
\makeatletter
\newcommand{\strikeword}[1]{\leavevmode\setbox\z@\hbox{#1}\rlap{\raisebox{0.55ex}{\vrule height0.4pt depth0pt width\wd\z@}}\box\z@}

\def\soutwords#1 {\ifx\relax#1\else\strikeword{#1} \expandafter\soutwords\fi}
\makeatother

\theoremstyle{plain} 
\newtheorem{theorem}{Theorem}
\newtheorem{lemma}[theorem]{Lemma} 
 
\newtheorem{proposition}[theorem]{Proposition}

\newtheorem{example}[theorem]{Example}
\theoremstyle{remark}
\newtheorem*{remark}{Remark}
\theoremstyle{definition}
\newtheorem{definition}[theorem]{Definition}

\newcommand{\K}{\mathcal{K}}
\newcommand{\mF}{\mathcal F}
\newcommand{\F}{\mF}
\newcommand{\N}{\mathbb{N}}

\newcommand{\on}{\operatorname}
\title{On the equicontinuity of the Hutchinson operator}
\author{Filip Strobin and Piotr Szuca}

\address{F. Strobin: Institute of Mathematics, University of Gdańsk, Poland; ORCID: 0000-0002-8671-9053}
\email{filip.strobin@ug.edu.pl}
\address{P. Szuca: Institute of Mathematics, University of Gdańsk, Poland; ORCID: 0000-0002-7833-011X}
\email{piotr.szuca@ug.edu.pl}

   % MSC

\begin{document}

\begin{abstract} {We show that the Hutchinson operator of an Iterated Function System with an attractor is equicontinuous if and only if it is continuous; in particular, every finite Iterated Function System with an attractor has equicontinuous Hutchinson operator. We also characterize equicontinuity of Hutchinson operators generated by a natural class of continuous multifunctions in terms of continuity of their limit maps.}
%We prove that the Hutchinson operators of a wide class of Iterated Function Systems are equicontinuous. In particular, every Iterated Function System with an attractor has equicontinuous Hutchinson operator.
\end{abstract}

\maketitle

\section{Introduction}
One of the problems undertaken in the theory of dynamical systems is that of equicontinuity. 
Namely, if $f$ is a selfmap of a metric space (thus $(X,f)$ is a dynamical system), then $f$ is said to be pointwise or uniformly equicontinuous, 
if the family of all iterations $\{f^{(n)}:n\in\N\}$ is pointwise or uniformly equicontinuous (see Def.~\ref{def:eq} and~\ref{def:eq-selfmap} below). 
In the literature there can be found deep discussion on the problem and various sufficient and/or necessary conditions for equicontinuity of $f$---cf. ~\cite{MR4241503}, \cite{MR1182256}
for one-dimensional dynamical systems, or \cite{MR1990578} for a general setting 
(in Section \ref{sec:1} we give a short exposition of them).

The aim of the paper is to study the problem 
for special dynamical systems---Hutchinson operators $H_\F$ generated by iterated function systems $\F$.
A classical Hutchinson-Barnsley theorem states that finite IFSs that consist of Banach contractions admit attractors, which are contracting fixed points of Hutchinson operators. For such IFSs it is easy to see (cf. Example \ref{ex:111}) that $H_\F$ are uniformly equicontinuous. 
However, there are highly noncontractive IFSs with attractors, hence a question is to determine if IFSs with attractors are equicontinuous 
(note that there are maps with contracting fixed points which are not equicontinuous, 
see Example~\ref{example:com}). 
Some positive answers were given under some additional assumptions 
(cf.~\cite{MR3860435}, \cite{MR2251715}, \cite{Sarizadeh}, \cite{Sarizadeh2}, \cite{Sarizadeh3} {or \cite{FK2020}). The main result of this paper shows that for IFSs with attractors equicontinuity of the Hutchinson operator is equivalent to its continuity; in particular, the answer is positive for all finite IFSs. In fact, we formulate it in a bit wider frame---for maps generated by continuous multifunctions with good attracting properties.}

In Section 2 we give basic definitions and several known facts on equicontinuity. We recall basics of the IFS theory of fractals and present some examples which show that the question on equicontinuity of Hutchinson operators is not completely trivial. Finally, we state the announced result on equicontinuity of Hutchinson operators of IFSs with attractors (Theorem~\ref{thm:main}).

In Section 3 we recall the definition of a weakly Picard operator and present a generalization of a Hutchinson operator involving a multifunction in the place of an IFS. 
We formulate Theorem \ref{thm:cont-eqq} which is the main result of the paper.
It deals with certain Hutchinson operators defined via multifunctions which are weakly Picard in a strong sense. 
As an immediate corollary we obtain above-mentioned Theorem~\ref{thm:main}.

{In the original version of this paper Theorem~\ref{thm:main} was stated for arbitrary IFSs with attractors. We overlooked that the Hutchinson operator of an infinite IFS need not be continuous, i.e.\ need not be generated by a continuous multifunction, so that Theorem~\ref{thm:cont-eqq} does not apply to it. This was pointed out by the anonymous referee, to whom we owe Example~\ref{ex:referee} and the observation that for finite IFSs Theorem~\ref{thm:main} follows from \cite{Sarizadeh2}. We are grateful to the referee for these remarks{, which substantially improved the paper}.}

%%%%%%%%%%%%%%%%%%%%%%%%%%%%%%%%%%%%%%%%%%%%%%%%%%

\section{Preliminaries}\label{sec:1}
Throughout this section we assume that $(X,d_X)$ and $(Y,d_Y)$ are metric spaces.

\subsection{Pointwise and uniform equicontinuity} 

\begin{definition}\label{def:eq}
Let $\F=\{f_i:i\in I\}$ be a family of maps $X\to Y$.\\
We say that $\F$ is \emph{pointwise equicontinuous}, if
\begin{equation}\label{eq:aaaaqq}
\forall_{x\in X}\;\forall_{\varepsilon>0}\;\exists_{\delta>0}\;\forall_{i\in I}\;\forall_{y\in X}\;d_X(x,y)<\delta\;\Rightarrow\;d_Y(f_i(x),f_i(y))<\varepsilon.
\end{equation}
We say that $\F$ is \emph{uniformly equicontinuous}, 
if
\begin{equation}\label{eq:qq}
\forall_{\varepsilon>0}\;\exists_{\delta>0}\;\forall_{i\in I}\;\forall_{x,y\in X}\;d_X(x,y)<\delta\;\Rightarrow\;d_Y(f_i(x),f_i(y))<\varepsilon.
\end{equation}
\end{definition}

The following lemma is folklore.

\begin{lemma}\label{le:eqcompact}
Uniform equicontinuity of $\F$ implies pointwise equicontinuity of $\F$.
Moreover, these notions are equivalent provided that $X$ is compact.
\end{lemma}
Below we give some natural examples of equicontinuous maps.
We use the notation
$$\F^{\infty}:=\{f_{i_1}\circ\cdots\circ f_{i_n}:n\in\N,\;i_1,\ldots,i_n\in I\},$$
if $\F=\{f_i:i\in I\}$ consists of selfmaps of $X$.

\begin{example}\label{ex:111}$\;$\\
(1) If $\F=\{f_i:i\in I\}$ consists of Lipschitz maps 
with $\sup\{\on{Lip}(f_i):i\in I\}<\infty$, then $\F$ is uniformly equicontinuous.\\
(2) If $\F=\{f_i:i\in I\}$ consists of nonexpansive selfmaps of a metric space $X$, 
i.e., $\on{Lip}(f_i)\leq 1$ for all $i\in I$, 
then also the family of all finite compositions $\F^{\infty}$ 
consists of nonexpansive maps and hence $\F^{\infty}$ 
is also uniformly equicontinuous. 
\end{example}

For our study we will need the following lemma, which shows that equicontinuity on compact sets implies pointwise equicontinuity. 
Let us say that a family $\F=\{f_i\colon i\in I\}$ is
 \emph{uniformly equicontinuous on a set $K\subset X$}, if the family of restrictions
$\{f_{i\restriction K}\colon i\in I\}$ is uniformly equicontinuous, 
i.e.
\begin{equation}\label{eq:ueqY}
\forall_{\varepsilon>0}\;\exists_{\delta>0}\;\forall_{i\in I}\;\forall_{x,y\in K}\;d_X(x,y)<\delta\;\Rightarrow\;d_Y(f_i(x),f_i(y))<\varepsilon.
\end{equation}
Similarly we understand \emph{pointwise equicontinuity of $\F$ on a set $K$}.

\begin{lemma}\label{lem:newkey}
Let $\mathcal{F}=\{f_i:i\in I\}$.
The following conditions are equivalent:
\begin{enumerate}
    \item $\F$ is pointwise equicontinuous on each compact set;\label{lem:newkey:1a}
    \item $\F$ is uniformly equicontinuous on each compact set;\label{lem:newkey:1b}
    \item $\F$ is pointwise equicontinuous.\label{lem:newkey:2}
\end{enumerate}
\end{lemma}
\begin{proof}
We show only the implication ``(\ref{lem:newkey:1a})$\Rightarrow$(\ref{lem:newkey:2})''. Other implications are obvious (see Lemma~\ref{le:eqcompact}).

Suppose that $\mathcal{F}$ is not pointwise equicontinuous. Then there exists $x_0\in X$ and $\varepsilon_0$ such that
$$
\forall_{\delta>0}\;\exists_{x\in X}\;\exists_{i\in I}\;(d_X(x_0,x)<\delta\;\mbox{and}\;d_Y(f_i(x_0),f_i(x))\geq\varepsilon_0).
$$
From this we have
\begin{equation}\label{eq:newkey}
\forall_{k\in\N}\;\exists_{x_k\in X}\;\exists_{i_k\in I}\;(d_X(x_0,x_k)<\frac{1}{k}\;\mbox{and}\;d_Y(f_{i_k}(x_0),f_{i_k}(x_k))\geq \varepsilon_0).
\end{equation}
Now the set $K:=\{x_0\}\cup\{x_k:k\in\N\}$ is compact. 
Hence we can find $\delta>0$ such that for $y\in K$ with $d_X(x_0,y)<\delta$, 
for any $i\in I$, we have $d_Y(f_i(x_0),f_i(y))<\varepsilon_0$. 
In particular, if $k\in\N$ is such that $\frac{1}{k}<\delta$, 
then $d_Y(f_{i_k}(x_0),f_{i_k}(x_k))<\varepsilon_0$, 
which contradicts~(\ref{eq:newkey}).
\end{proof}

%%%%%%%%%%%%%%%%%%%%%%%%%%%%%%%%%%%%%%%%%%%%%%%

\subsection{Equicontinuity of a selfmap}

\begin{definition}\label{def:eq-selfmap}
We say that $f:X\to X$ is \emph{pointwise} [\emph{uniformly}]\emph{ equicontinuous}, if the family of iterations $\{f^{(n)}:n\in\N\}$ is pointwise [uniformly, respectively] equicontinuous.
\end{definition}

A problem of (pointwise or uniform) equicontinuity of a map has been studied in literature. In particular, several necessary, sufficient or equivalent conditions have been found for certain spaces. 
For example, for a compact space $X$, $f$ is uniformly equicontinuous iff $f$ is nonexpansive w.r.t. some admissible metric on $X$ (see \cite{Levitan}). We will also recall a recent theorem which gathers together and extends several earlier observations for a quite general class of compact spaces. Note that a finite tree is a certain dendrite (necessarily compact), and that each closed and bounded interval is an example of such a space (for definition, we refer the reader to \cite{MR4241503}).  We denote by $Fix(f)$ the set of fixed points of $f$ and $Per(f)$ is the set of periodic points of $f$; the closure $\overline{A}$ is considered w.r.t. the Tychonoff (or pointwise) topology in the compact product space $X^X$.
\begin{theorem} [\hbox{\cite[Thm.~1.8]{MR4241503}}]\label{thm:bc}
If $X$ is a finite tree and $f:X\to X$, 
then the following conditions are equivalent:
\begin{enumerate}
\item $f$ is uniformly (equivalently, pointwise) equicontinuous;
\item there is an $n\in\mathbb{N}$ such that restriction of $f^{(n)}$ to 
  $\bigcap_{n\in\mathbb{N}}f^{(n)}(X)$ is the identity map;
\item there is an $n\in\mathbb{N}$ such that 
 $\textrm{Fix}(f^{(n)})=\bigcap_{n\in\mathbb{N}}f^{(n)}(X)$;
\item $\textrm{Per}(f)=\bigcap_{n\in\mathbb{N}}f^{(n)}(X)$;
\item there is no arc $A\subset X$ such that for some $n\in\mathbb{N}$, we have $A\subsetneq f^{(n)}(A)$;
\item for every $n\in\mathbb{N}$, the set $\textrm{Fix}(f^{(n)})$ is connected;
\item the set $\textrm{Per}(f)$ is connected;\label{thm:bc:per}
\item every $f'\in\bigcap_{n=1}^\infty \overline{\{f^{(m)}:m\geq n\}}$ is continuous;
\item there exists $f'\in\bigcap_{n=1}^\infty \overline{\{f^{(m)}:m\geq n\}}$ which is continuous.
\end{enumerate}
\end{theorem}
We refer the reader to \cite{MR4241503} 
for a brief exposition of the earlier studies of the topic.

\subsection{Iterated function systems, their Hutchinson operators and attractors}

For a metric space $(X,d)$, by $\K(X)$ we denote the hyperspace of all nonempty and compact subsets of $X$ endowed with the Hausdorff metric $h$.

\begin{definition}\label{def:IFS}
By an \emph{iterated function system} (IFS) we mean a family $\mF=\{f_i:i\in I\}$ of continuous selfmaps of $X$ so that for every $K\in\K(X)$, the set $\overline{\bigcup_{i\in I}f_i(K)}\in\K(X)$. Then the map $H_\mF:\K(X)\to\K(X)$ defined by
$$
H_\mF(K):=\overline{\bigcup_{i\in I}f_i(K)}
$$
is called \emph{the Hutchinson operator of $\mF$}.
\end{definition}

\begin{definition}
Let $\mF$ be an IFS. A (necessarily unique) set $A_\mF\in\K(X)$ that satisfies the following conditions:
\begin{itemize}
\item[(i)] $A_\mF=H_\mF(A_\mF)=\overline{\bigcup_{i\in I}f_i(A_\mF)}$;
\item[(ii)] for every $K\in\K(X)$, the sequence of iterations $(H_\mF^{(n)}(K))$ converges to $A_\mF$ w.r.t. the Hausdorff metric $h$,
\end{itemize}
is called \emph{an attractor of $\mF$}.
\end{definition}

In other words, an attractor is a contracting fixed point of the Hutchinson operator {(which is automatically a fixed point when $H_\mF$ is continuous)}, where we say that (cf.~Def.~\ref{def:WPO}):
\begin{definition}\label{def:CFP}
$x_*$ is \emph{a contracting fixed point of $f\colon X\to X$}, if  for every $x\in X$, the limit $\lim_{n\to\infty}f^{(n)}(x)=x_*$.
\end{definition}

The following well-known Hutchinson theorem (see \cite{Hut})
states that finite contractive IFSs on complete metric spaces
admit attractors.

\begin{theorem}
Let $\mF=\{f_1,\ldots,f_k\}$ be an IFS on a complete metric space $X$ consisting of Banach contractions (that is, their Lipschitz constants $\on{Lip}(f_i)<1$). Then $\mF$ admits an attractor. 
\end{theorem}

This result can be proven with a help of a Banach fixed point theorem. Namely, it can be easily shown that
$$\on{Lip}(H_\mF)\leq\max\{\on{Lip}(f_i)\colon i=1,\ldots,k\}<1,$$ 
so $H_{\mF}$ is a Banach contraction, and hence it has a contracting fixed point which is exactly an attractor $A_\mF$. Hutchinson's result can be extended in many ways, 
for example for possibly infinite IFSs, 
or IFSs that consist of maps satisfying weaker contractive conditions. 
However, there are IFSs with attractors which are highly noncontractive, equivalently---which are not equicontinuous themselves:

\begin{example}[\hbox{\cite[Example 2]{Lesniak2024}}]\label{example:1}
\emph{
Let $\F=\{f_1,f_2,f_3\}$  be an IFS on $[0,1]$, where
$$
f_1(x):=\sqrt{x},\;\;\;
f_2(x):=\sqrt{1-x},\;\;\;
f_3(x):=\left\{\begin{array}{ccc}
  \frac{1}{16}&\mbox{for}&x<\frac{1}{8};\\
  \frac{7}{2}x-\frac{3}{8}&\mbox{for}&\frac{1}{8}\leq x\leq\frac{1}{4};\\
  \frac{1}{2}&\mbox{for}&x>\frac{1}{4}.
\end{array}\right.
$$
Then $[0,1]$ is the attractor of $\mF$, but maps $f_1,f_2,f_3$ are highly non-nonexpansive, 
i.e.~for any  admissible metric $d$ on $[0,1]$, we have $\on{Lip}(f_i)>1$ for $i=1,2,3$. 
As mentioned earlier, this is equivalent to saying that $f_1,f_2$ and $f_3$ are not uniformly equicontinuous.
}
\end{example}

For more examples of {highly non-contractive} IFSs we refer the reader to~\cite{Lesniak2024}.%{\footnote{{\color{red}\sout{Corrected: the previous version read ``nonexpansive IFSs''; the IFSs in \cite{Lesniak2024} are, on the contrary, not nonexpansive with respect to any admissible metric.}}}}

The paper is devoted to the question on equicontinuity of Hutchinson operators of IFSs. 
Our motivations came from the question on Hutchinson operators of IFSs with attractors.
Theorem \ref{thm:cont-eqq} (formulated in a more abstract form) implies that {for IFSs with attractors on metric spaces, continuity of the Hutchinson operator is equivalent to its pointwise equicontinuity, i.e., we get the following result.}

\begin{theorem}\label{thm:main}{
If an IFS $\mF$ on a metric space admits an attractor, then $H_\mF$ is pointwise equicontinuous if and only if $H_\mF$ is continuous. In particular, the Hutchinson operator {of a finite} IFS with an attractor is pointwise equicontinuous.}
\end{theorem}

{If $\mF$ is finite, then $H_\mF$ is continuous (cf.\ the remark following (\ref{eq:genH})). For infinite IFSs the existence of an attractor does not imply continuity of $H_\mF$, see Example~\ref{ex:referee}; for Hutchinson operators of infinite IFSs see \cite{GLJ2005}, \cite{Lesniak2015}, \cite{FK2020}.}

A natural question arises if Theorem \ref{thm:main} is nontrivial. 
In the literature we found several weaker versions of Theorem \ref{thm:main}, restricted to compact spaces. In \cite{Sarizadeh2}, using different methods, the author proved the assertion for the case when $X$ itself is an attractor of $\F$ (see also \cite{Sarizadeh}) and in \cite{MR2251715} for IFSs on compact spaces which are \emph{stable} or, in different notation, topologically contractive; equivalently---weakly contractive under some remetrization (see~\cite{Banakh}). Our Lemma \ref{le:X-eqq} can be considered as a slightly stronger result than that of~\cite{Sarizadeh2}.

{In particular, for finite IFSs Theorem~\ref{thm:main} follows from \cite[Theorem A(vi)]{Sarizadeh2} and a standard compactness argument. Theorem~\ref{thm:cont-eqq} goes further: it deals with Hutchinson operators generated by arbitrary continuous multifunctions, allows the limit map $H^*$ to be non-constant, and gives an equivalence. For attractors and minimal invariant sets of non-contractive IFSs see also \cite{BV2013} and \cite{BGMS2017}.}

First we show that there are maps with contracting fixed points which are not pointwise equicontinuous.

\begin{example}\label{example:com}
\emph{Let $S$ be the unit circle in the 
complex plane $\mathbb{C}$, and let $f:S\to S$ be defined by
$$
f(e^{2\pi i t}) := e^{2\pi i\sqrt{t}}
$$
for $t\in[0,1]$. Then $1$ is a contracting fixed point of $f$ but $f$ is not uniformly (pointwise, equivalently) equicontinuous as 
$$\big|e^{2\pi i\big(\frac{1}{2}\big)^{2^n}}-1\big|\to 0,$$
and
$$
\big|f^{(n)}\big(e^{2\pi i\big(\frac{1}{2}\big)^{2^n}}\big)-f^{(n)}(1)\big|=|-1-1|=2.
$$
}
\end{example}
If $\mF$ consists of nonexpansive maps, then automatically its Hutchinson operator $H_\mF$ is nonexpansive, so by Example \ref{ex:111}(2), it is uniformly equicontinuous. 
In fact, the following result seems to be folklore,
but we give a short proof for the convenience of the reader.

\begin{proposition}
Assume that $\mF=\{f_i:i\in I\}$ is an IFS so that the family 
$\mF^{\infty}$ 
is pointwise equicontinuous [uniformly equicontinuous]. 
Then its Hutchinson {operator} $H_\mF$ is pointwise equicontinuous [uniformly equicontinuous, respectively].%{\footnote{{\color{red}\sout{Corrected: the previous version read ``Hutchinson attractor''.}}}}
\end{proposition}

\begin{proof}
We will only prove the case of pointwise equicontinuity, which is a bit more technically complicated. Assume that $\mF^\infty$ is pointwise equicontinuous. 
Take any $K\in \K(X)$ and $\varepsilon>0$. 
Then for every $x\in K$, there is $\delta_x>0$ 
such that for every $f\in\mF^\infty$, if $d(y,x)<\delta_x$, then $d(f(x),f(y))<\frac{\varepsilon}{2}$. 
Since $K$ is compact, we can find $x_1,\ldots,x_k\in K$ so that 
$$K\subset\bigcup_{i=1}^k B\big(x_i,\frac{\delta_{x_i}}{2}\big).$$
Let $\delta:=\tfrac12\min\{\delta_{x_1},\ldots,\delta_{x_k}\}$, take any $n\in\N$ and let $D\in\K(X)$ be such that $h(K,D)<\delta$. 
If $x\in H^{(n)}_{\F}(K)$, then $x=f_{i_1}\circ\cdots\circ f_{i_{n}}(x')$ for some $x'\in K$ and $i_1,\ldots,i_{n}\in I$. Since $h(K,D)<\delta$, we can find $y'\in D$ with $d(x',y')<\delta$. Now since $x'\in K$, there is $i=1,\ldots,k$ so that $x'\in B(x_i,\frac{\delta_{x_i}}{2})$, and thus $x',y'\in B(x_i,\delta_{x_i})$. In particular, setting $f:=f_{i_1}\circ\cdots\circ f_{i_{n}}$, we have
$$
\inf\{d(x,y)\colon y\in H^{(n)}_\mF(D)\}\leq d(f(x'),f(y'))\leq d(f(x'),f(x_i))+d(f(x_i),f(y'))<\frac{\varepsilon}{2}+\frac{\varepsilon}{2}=\varepsilon,
$$
and hence
$$
\sup\{\inf\{d(x,y)\colon y\in H^{(n)}_\mF(D)\}\colon x\in H^{(n)}_\mF(K)\}\leq\varepsilon.
$$
In the same way we show that
$$
\sup\{\inf\{d(x,y)\colon x\in H^{(n)}_\mF(K)\}\colon y\in H^{(n)}_\mF(D)\}\leq\varepsilon,
$$
which gives $h(H^{(n)}(K),H^{(n)}(D))\leq \varepsilon$.
\end{proof}

Note that equicontinuity of each element of $\mF$ does not imply equicontinuity of $\mF^\infty$:

\begin{example}\label{ex:eq}\emph{
   Let $f_1,f_2\colon [0,1]\to[0,1]$ be defined by
 $$f_1(x):=\frac{1}{2}x+\frac{1}{4}, \;\;\;\;
 f_2(x):=\left\{\begin{array}{ll}
        \frac{3}{4} & \textup{for\ }x\leq\frac{5}{8}; \\
        4x-\frac{7}{4} & \textup{for\ } \frac{5}{8}<x<\frac{11}{16}; \\
        1 & \textup{for\ }x\geq\frac{11}{16}.
    \end{array}\right.$$
Then
    $$(f_1\circ f_2)(x)=\left\{\begin{array}{ll}
        \frac{5}{8} & \textup{for\ }x\leq\frac{5}{8}; \\
        2x-\frac{5}{8} & \textup{for\ } \frac{5}{8}<x<\frac{11}{16}; \\
        \frac{3}{4} & \textup{for\ }x\geq\frac{11}{16}.
    \end{array}\right.$$
$f_1,f_2$ are uniformly (equivalently, pointwise) equicontinuous ($f_2^{(2)}(x)=1$) 
but $f_1\circ f_2$ is not uniformly (equivalently, pointwise) equicontinuous
(this follows easily from Theorem \ref{thm:bc}), which implies that $\mF^\infty$ is not uniformly (equivalently, pointwise) equicontinuous.
}
\end{example}

Observe also that equicontinuity of $\mF^\infty$ is not a necessary
nor a sufficient condition for the existence of an attractor:

\begin{example}
\emph{
Consider the IFS $\mF=\{f_1,f_2,f_3\}$ from Example \ref{example:1}. 
The maps $f_1,f_2,f_3$ are not uniformly equicontinuous, 
and in particular, $\mF^\infty$ is not uniformly equicontinuous. 
(Recall again that uniform equicontinuity of $f$ on a compact space is equivalent to the existence of an admissible metric making $f$ nonexpansive)
}
\end{example}

\begin{example}
\emph{
Let
$$
X=\{-1\}\cup \big\{x_n\colon n\in\mathbb{Z}\}\cup\{1\},
$$
where $(x_n)_{n\in{\mathbb{Z}}}$ satisfies $\lim_{n\to-\infty}x_n=-1$ and $\lim_{n\to\infty}x_n=1$ and $x_n<x_{n+1}$ for $n\in\mathbb{Z}$,
and consider $X$ as a subspace of $\mathbb{R}$. Then $X$ is compact.\\
Define $f_1,f_2,f_3\colon X\to X$ by
$$
f_1(x)=\left\{\begin{array}{ccc}
x&\mbox{if}&x\in\{-1,1\};\\
x_{n+1}&\mbox{if}&x=x_n,
\end{array}\right.
$$
$$
f_2(x)=\left\{\begin{array}{ccc}
-1&\mbox{if}&x\in\{-1,x_0,1\};\\
x_{n+1}&\mbox{if}&x=x_n\;\mbox{for}\;n<0;\\
x_{-n}&\mbox{if}&x=x_n\;\mbox{for}\;n>0,
\end{array}\right.
$$
$$
f_3(x)=x_0\;\;\mbox{for}\;\;x\in X.
$$
Then $f_1,f_2,f_3$ are continuous and $X$ is the attractor of $\F=\{f_1,f_2,f_3\}$. 
This follows from $H_{\F}^{(n)}(\{x_0\})\to X$, which is a consequence of the fact that for $n\in\N$, 
$$
{\{x_{-n+1},\ldots,x_{-1},x_0,x_1,\ldots,x_n\}}\subset H_{\F}^{(n)}(\{x_0\}).
$$%\footnote{{\color{red}\sout{there was a typo in the previous version}}}
Observe that the family 
$$\{f_{i_1}\circ\cdots\circ f_{i_n}\colon n\in\N,\;i_1,\ldots,i_n=1,2\}$$
is not pointwise (equivalently uniformly) equicontinuous. 
Indeed, we have
$$
|x_{-n}-(-1)|\to0
$$
and
$$
|f_{i_1}\circ\cdots\circ f_{i_n}(x_{-n})-f_{i_1}\circ\cdots\circ f_{i_n}(-1)|=|x_0-(-1)|.
$$
}
\end{example}

\begin{example}
\emph{
Let $\F=\{\textup{Id}\}$, where  $\textup{Id}\colon[0,1]\to[0,1]$ is the identity function. Then $\F^\infty$ is equicontinuous but $\F$ does not
possess an attractor.
}
\end{example}

%%%%%%%%%%%%%%%%%%%%%%%%%%%%%%%%%%%%%%%%%%%%%%%%%%%%%%%%%%%%%%%%%%%%

\section{Main results}

Throughout this section we assume that $(X,d)$ is a metric space.

%%%%%%%%%%%%%%%%%%%%%%%%%%%%%%%%%%%%%%%%

\subsection{Weakly Picard operators}

%Below we recall a bit weaker notion of contracting fixed point (cf.~Def.~\ref{def:CFP}):
\begin{definition}\label{def:WPO}
We say that a continuous selfmap $f\colon X\to X$ is \emph{weakly Picard}, 
if for every $x\in X$, the limit $\lim_{n\to\infty}f^{(n)}(x)$ exists. 
Each weakly Picard operator $f$ generates the map $f^*\colon X\to X$ by
$$
f^*(x):=\lim_{n\to\infty}f^{(n)}(x).
$$
\end{definition}

Observe that each $f^*(x)$ for $x\in X$ is a fixed point of $f$. Indeed, by continuity of $f$, we have
$$
f(f^*(x))=f(\lim_{n\to\infty}f^{(n)}(x))=\lim_{n\to\infty}f^{(n+1)}(x)=f^*(x).
$$
The next lemma shows that for weakly Picard operators, equicontinuity implies continuity of $f^*$. We say that $f\colon X\to X$ is \emph{pointwise equicontinuous} [\emph{uniformly equicontinuous}] \emph{on a set $Y\subset X$}, if the family of restrictions of iterations 
$\{f^{(n)}\colon n\in\N\}$
is pointwise equicontinuous [uniformly equicontinuous, respectively]
on a set $Y$.

\begin{lemma}\label{lem:eqcont1}
If $f\colon X\to X$ is weakly Picard and is pointwise equicontinuous [uniformly equicontinuous] on a set $Y\subset X$, 
then the restriction $f^*_{\restriction Y}:Y\to X$ is continuous [uniformly continuous, respectively].
\end{lemma}

\begin{proof}
Assume that $f$ is pointwise equicontinuous on $Y$ and take any $x\in Y$ and $\varepsilon>0$. 
Then choose $\delta>0$ such that for every $n\in\N$ and $y\in Y$, if $d(x,y)<\delta$, then we have $d(f^{(n)}(x),f^{(n)}(y))<\varepsilon$. Taking the limit in this inequality, we arrive to $d(f^*(x),f^*(y))\leq \varepsilon$. 
This gives continuity of $f^*_{\restriction Y}$.
\end{proof}

The next lemma characterizes uniform equicontinuity for a wide class of maps.

\begin{lemma}\label{lem:wPic1}
Assume that $f$ is weakly Picard and $Y\subset X$ is nonempty and compact. % and such that $f^*$ is continuous.
Then $f$ is uniformly (equivalently, pointwise) equicontinuous on $Y$ if and only if the restriction $f^*_{\restriction Y}$ is uniformly continuous and the following holds: 
\begin{equation}\label{eq:eqq:eksperyment}
\forall_{\varepsilon>0}\;\exists_{n_0\in\N}\;\forall_{n\geq n_0}\;\forall_{x\in Y}\;d(f^{(n)}(x),f^*(x))<\varepsilon.
\end{equation}
\end{lemma}

\begin{proof}
Assume first that $f$ is uniformly equicontinuous on $Y$. 
By Lemma \ref{lem:eqcont1}, $f^*_{\restriction Y}$ is uniformly continuous. 
{Fix any $\varepsilon>0$ and choose $\delta>0$ such that $d(f^{(n)}(x),f^{(n)}(y))<\frac{\varepsilon}{3}$ and $d(f^*(x),f^*(y))<\frac{\varepsilon}{3}$ for all $n\in\N$ and $x,y\in Y$ with $d(x,y)<\delta$. Since $Y$ is compact, there are $x_1,\ldots,x_k\in Y$ with $Y\subset\bigcup_{i=1}^kB(x_i,\delta)$, and there is $n_0\in\N$ such that $d(f^{(n)}(x_i),f^*(x_i))<\frac{\varepsilon}{3}$ for all $n\geq n_0$ and $i=1,\ldots,k$. Now let $n\geq n_0$ and $x\in Y$, and choose $i$ with $d(x,x_i)<\delta$. Then
$$d(f^{(n)}(x),f^*(x))\leq d(f^{(n)}(x),f^{(n)}(x_i))+d(f^{(n)}(x_i),f^*(x_i))+d(f^*(x_i),f^*(x))<\varepsilon.$$
This gives (\ref{eq:eqq:eksperyment})}.%\footnote{{\color{red}\sout{In the previous version this implication was proved by a Baire-type argument which applied the uniform equicontinuity of $f$ on $Y$ to the pair of points $f^{(n_0)}(x)$, $f^*(x)$; these points need not belong to $Y$ when $Y$ is not invariant, so that argument was incomplete. The present argument is the standard fact that an equicontinuous sequence converging pointwise on a compact set converges uniformly.}}}}

Now assume that $f^*_{\restriction Y}$ is uniformly continuous 
and~(\ref{eq:eqq:eksperyment}) holds, choose any $\varepsilon>0$ 
and let $n_0$ be as in (\ref{eq:eqq:eksperyment}) 
chosen for $\frac{\varepsilon}{3}$. 
Since $f^{(1)},\ldots,f^{(n_0-1)}$ are uniformly continuous on $Y$ (as $Y$ is compact) and $f^*_{\restriction Y}$ is uniformly continuous, we can find $\delta>0$ 
such that $d(f^*(x),f^*(y))<\varepsilon$ and $d(f^{(k)}(x),f^{(k)}(y))<\varepsilon$ for $x,y\in Y$ 
with $d(x,y)<\delta$ and all $k=1,\ldots,n_0-1$. 
%Moreover, by uniform continuity of $f^*$, we may assume that $d(f^*(x),f^*(y))<\frac{\varepsilon}{3}$
%for $x,y$ with $d(x,y)<\delta$.
Now choose $x,y\in Y$ with $d(x,y)<\delta$ and take any $n\in\N$. 
If $n\leq n_0-1$, then $d(f^{(n)}(x),f^{(n)}(y))<\varepsilon$ 
by the choice of $\delta$. If $n\geq n_0$, then 
$$d(f^{(n)}(x),f^{(n)}(y))\leq d(f^{(n)}(x),f^*(x))+d(f^*(x),f^*(y))+d(f^*(y),f^{(n)}(y))<\varepsilon.$$ 
This ends the proof.
\end{proof}

%%%%%%%%%%%%%%%%%%%%%%%%%%%%%%%%%%%%%%%%%%%%%%%%%%%%%%%

\subsection{Hutchinson operators generated by multifunctions}
\begin{definition}\label{def:genH}
If $F:X\to\K(X)$ is a continuous multifunction, then the map $H_F:\K(X)\to \K(X)$ defined by
\begin{equation}\label{eq:genH}
H_F(K):=\bigcup_{x\in K}F(x)
\end{equation}
will be called \emph{the Hutchinson operator generated by $F$}.\\
We say that $H:\K(X)\to\K(X)$ is \emph{a Hutchinson operator} if it is the Hutchinson operator generated by some continuous multifunction $F:X\to\K(X)$. 
\end{definition}

\begin{remark}
By \cite[Proposition 1 and Theorem 1]{MR3446001}, a Hutchinson operator is well defined and is a continuous map.
\end{remark}
\begin{remark}
Hutchinson operators of a wide class of IFSs are of the form (\ref{eq:genH}). 
To see it, assume that $\mF=\{f_i:i\in I\}$ is an IFS satisfying the following properties:
\begin{itemize}
\item[(i)] for every $x\in X$, the set $\overline{\{f_i(x):{i\in I}\}}$ is compact;%{\footnote{{\color{red}\sout{Corrected: in the previous version the set was written as $\overline{\{f_i(x):x\in X\}}$, i.e.\ with the wrong index.}}}}
\item[(ii)] the family $\{f_i:i\in I\}$ is pointwise equicontinuous.
\end{itemize}
Then its Hutchinson operator $H_\mF$ is of the form
$$
\forall_{K\in\K(X)}\;H_\mF(K)=\bigcup_{x\in K}F(x){=H_F(K)}
$$ 
for
$$F(x):=\overline{\{f_i(x):i\in I\}}.$$
Indeed, the map $F:X\to\K(X)$ is well-defined (by (i)) and continuous (by (ii)), and we then have
$$
H_\mF(K)=\overline{\bigcup_{i\in I}f_i(K)}=\overline{\bigcup_{i\in I}\bigcup_{x\in K}\{f_i(x)\}}=
\overline{\bigcup_{x\in K}\{f_i(x):i\in I\}}=$$ $$=\overline{\bigcup_{x\in K}\overline{\{f_i(x):i\in I\}}}=
\bigcup_{x\in K}\overline{\{f_i(x):i\in I\}}=\bigcup_{x\in K}F(x).
$$
\end{remark}

\begin{lemma}\label{lem:hhelpp}
If $H$ is a Hutchinson operator, then for every $K\in\K(X)$ and $n\in\N$,
\begin{equation}\label{eq:help}
H^{(n)}(K)=\bigcup_{x\in K}\; H^{(n)}(x),
\end{equation}
where $H^{(n)}(x):=H^{(n)}(\{x\})$.
\end{lemma}
\begin{proof}
For $n=1$ there is nothing to prove. Assume that (\ref{eq:help}) holds for some $n\in\N$.
On the one hand
$$
H^{(n+1)}(K)=H^{(n)}(H(K))=\bigcup_{y\in H(K)}H^{(n)}(y)=\bigcup_{x\in K}\bigcup_{y\in F(x)}H^{(n)}(y),
$$
and on the other
$$
\bigcup_{x\in K}H^{(n+1)}(x)=\bigcup_{x\in K}H^{(n)}(H(x))=\bigcup_{x\in K}\bigcup_{y\in H(x)}H^{(n)}(y)=\bigcup_{x\in K}\bigcup_{y\in F(x)}H^{(n)}(y),
$$
and we get (\ref{eq:help}) for $n+1$.
\end{proof}

The above lemma implies that uniform equicontinuity of $H$ can be reduced to uniform equicontinuity on the subspace of singletons. 
Below we identify subset $Y\subset X$ with $\{\{x\}:x\in Y\}\subset\K(Y)$.

\begin{lemma}\label{lem:sing}
Let $Y\subset X$.
A Hutchinson operator $H:\K(X)\to \K(X)$ is uniformly equicontinuous on $\K(Y)$ if and only if $H$ is uniformly equicontinuous on $Y$, i.e.,
\begin{equation}\label{eq:sing}
\forall_{\varepsilon>0}\;\exists_{\delta>0}\;\forall_{n\in\N}\;\forall_{x,y\in Y}\;d(x,y)<\delta\;\Rightarrow\;h(H^{(n)}(x),H^{(n)}(y))<\varepsilon.
\end{equation}
\end{lemma}

\begin{proof}
Clearly we only have to prove the implication ``$\Leftarrow$''.

Assume (\ref{eq:sing}) holds and take any $\varepsilon>0$. Then choose $\delta>0$ as in (\ref{eq:sing}) and fix any $n\in\N$ and $K,D\in \K(Y)$  with $h(K,D)<\delta$. 
Take $x\in H^{(n)}(K)$. By Lemma \ref{lem:hhelpp}, $x\in H^{(n)}(x')$ for some $x'\in K$. 
Consequently we can find $y'\in D$ with $d(x',y')<\delta$ and 
$$
\inf\{d(x,y)\colon y\in H^{(n)}(D)\}\leq \inf\{d(x,y)\colon y\in H^{(n)}(y')\}\leq h(H^{(n)}(x'),H^{(n)}(y'))<\varepsilon.
$$
In consequence, we obtain
$$
\sup\{\inf\{d(x,y):y\in H^{(n)}(D)\}\colon x\in H^{(n)}(K)\}\leq\varepsilon.
$$
In the same way we show that 
$$
\sup\{\inf\{d(x,y):y\in H^{(n)}(K)\}:x\in H^{(n)}(D)\}\leq\varepsilon,
$$
which gives $h(H^{(n)}(K),H^{(n)}(D))\leq\varepsilon.$
\end{proof}
%\begin{lemma}\label{le:eqq-cond}
%Let $X$ be a compact metric space and $H:\K(X)\to\K(X)$ be weakly Picard. Then $H$ is uniformly equicontinuous if and only if $H^*$ is continuous and 
%\begin{equation}\label{eq:final:eksperyment}
%\forall_{\varepsilon>0}\;\exists_{n_0\in\N}\;\forall_{n\geq n_0}\;\forall_{x\in X}\;h(H^{(n)}(x),H^*(x))<\varepsilon.
%\end{equation}
%\end{lemma}
%\begin{proof}
%Since for every $K\in\K(X)$, we have
%$$
%h(H^{(n)}(K),H^*(K))=h\big(\overline{\bigcup_{x\in K}H^{(n)}(x)},\overline{\bigcup_{x\in K}H^*(x)}\big)\leq\sup\{h(H^{(n)}(x),H^*(x)):x\in K\},
%$$
%condition (\ref{eq:final:eksperyment}) implies (\ref{eq:eqq:eksperyment}) and the result follows from Lemma \ref{lem:wPic1}.
%\end{proof}
\subsection{Self-attracting weakly Picard Hutchinson operators with continuous limit map are pointwise equicontinuous}

Finally we can state and prove the main result of the paper. We need a definition.

\begin{definition}
Let $H:\K(X)\to\K(X)$ be a weakly Picard Hutchinson operator. 
We say that $H$ \emph{satisfies condition} (T) (or is \emph{self-attracting}), 
{if for every $x\in X$, the set $H^*(\{x\})$ is a minimal invariant set of $H$, i.e., there is no nonempty compact $D\subsetneq H^*(\{x\})$ with $H(D)\subset D$.}
\end{definition}

{Since for a nonempty compact $D$ with $H(D)\subset D$, the set $H^*(D)=\bigcap_{n\in\N}H^{(n)}(D)$ is a nonempty compact subset of $D$ with $H(H^*(D))=H^*(D)$, it makes no difference whether minimality is understood with respect to sets $D$ satisfying $H(D)\subset D$ or $H(D)=D$. Our terminology follows \cite{Sarizadeh}; cf.\ also \cite{BV2013} and \cite{FK2020}, where minimality refers to sets with $H(D)=D$.}

{
\begin{lemma}\label{lem:T}
Let $H:\K(X)\to\K(X)$ be a weakly Picard Hutchinson operator. The following conditions are equivalent:
\begin{enumerate}
\item $H$ satisfies condition (T);\label{lem:T:1}
\item for every $x\in X$ and every nonempty compact $D\subset H^*(\{x\})$, we have $H^*(D)=H^*(\{x\})$;\label{lem:T:2}
\item for every $x\in X$ and every $a\in H^*(\{x\})$, we have $H^*(\{a\})=H^*(\{x\})$.\label{lem:T:3}
\end{enumerate}
\end{lemma}
\begin{proof}
``(\ref{lem:T:1})$\Rightarrow$(\ref{lem:T:2})''. Let $D\subset H^*(\{x\})$ be nonempty and compact. Since $H^*(\{x\})$ is a fixed point of $H$, we have $H^{(n)}(D)\subset H^{(n)}(H^*(\{x\}))=H^*(\{x\})$ for all $n$, so $H^*(D)\subset H^*(\{x\})$. As $H^*(D)$ is a nonempty compact set with $H(H^*(D))=H^*(D)$, (\ref{lem:T:1}) gives $H^*(D)=H^*(\{x\})$.
``(\ref{lem:T:2})$\Rightarrow$(\ref{lem:T:3})'' is obvious.
``(\ref{lem:T:3})$\Rightarrow$(\ref{lem:T:1})''. Let $D\subset H^*(\{x\})$ be nonempty compact with $H(D)\subset D$, and take $a\in D$. Then $H^{(n)}(\{a\})\subset D$ for all $n$, so $H^*(\{a\})\subset D$, and by (\ref{lem:T:3}) we get $H^*(\{x\})=H^*(\{a\})\subset D$.
\end{proof}}

Clearly, if $H^*$ is constant then (T) holds; but (T) also allows $H^*$ to be non-constant, 
e.g.~on a disjoint union $X=X_1\sqcup X_2$ of two invariant pieces with different attractors.

Our main result is the following.

\begin{theorem}\label{thm:cont-eqq}
Assume that $H\colon\K(X)\to\K(X)$ is a weakly Picard Hutchinson operator satisfying condition (T).
Then the following conditions are equivalent:
\begin{enumerate}
\item $H^*:\K(X)\to\K(X)$ is continuous;\label{thm:cont-eqq:1}
\item $H$ is pointwise equicontinuous.\label{thm:cont-eqq:2}
\end{enumerate}
\end{theorem}
{
Before we give a proof, let us observe that Theorem \ref{thm:cont-eqq} implies Theorem \ref{thm:main}. 
Indeed, let $\F=\{f_i:i\in I\}$ admit an attractor and let $H_\F$ be continuous. Put $F(x):=H_\F(\{x\})=\overline{\{f_i(x):i\in I\}}$ for $x\in X$. Since $x\mapsto\{x\}$ is an isometric embedding of $X$ into $\K(X)$, $F$ is a continuous multifunction. {As in the second remark after Definition~\ref{def:genH} (whose conditions (i) and (ii) were used there only to ensure that $F$ is a well-defined continuous multifunction), we see that $H_\F=H_F$, %\footnote{{\color{red}Conditions (i) and (ii) of that remark are not assumed here; they were used there only to ensure that $F$ is a well-defined continuous multifunction, which is here obtained from the isometric embedding. To make the reference precise one could write: ``As in the second remark after Definition~\ref{def:genH} (whose conditions (i) and (ii) were used there only to ensure that $F$ is a well-defined continuous multifunction), we see that $H_\F=H_F$,i.e., \ldots''.}},
i.e.,  %so $H_F(K)=\bigcup_{x\in K}F(x)$ is compact for every $K\in\K(X)$ by \cite{MR3446001}. Hence $H_\F(K)\subset H_F(K)$, and clearly $H_F(K)\subset H_\F(K)$.Thus
$H_\F$ is a Hutchinson operator in the sense of Definition~\ref{def:genH} generated by $F$}. Moreover, $H_\F$ is weakly Picard and the map $H_\F^*$ is a constant (hence continuous and satisfying condition (T)) map. 
Thus $H_\F$ is pointwise equicontinuous by Theorem~\ref{thm:cont-eqq}. The converse implication in Theorem~\ref{thm:main} is trivial.}

{For infinite IFSs the existence of an attractor does not imply continuity of $H_\F${, as the following example, which we owe to the anonymous referee, shows}.}

\begin{example}[{due to the referee}]\label{ex:referee}
\emph{
{Let $a_m:=\frac{1}{m+1}$ for $m\in\N$ and $X:=\{0,1\}\cup\{a_m:m\in\N\}$ with the metric inherited from $\mathbb{R}$. Then $X$ is compact and every point other than $0$ is isolated. For $m\in\N$ let $f_m\colon X\to X$ be given by $f_m(a_m):=1$ and $f_m(x):=0$ for $x\neq a_m$. Each $f_m$ is continuous, and for every $K\in\K(X)$,
$$
H_\F(K)=\overline{\bigcup_{m\in\N}f_m(K)}=\begin{cases}\{0\}, & K\subset\{0,1\},\\ \{0,1\}, & K\cap\{a_m:m\in\N\}\neq\emptyset,\end{cases}
$$
so $\F:=\{f_m:m\in\N\}$ is an IFS. Since $H_\F(\{0,1\})=\{0\}$, we get $H_\F^{(n)}(K)=\{0\}$ for all $K\in\K(X)$ and $n\geq2$, so $\{0\}$ is an attractor of $\F$. However, $\{a_m\}\to\{0\}$, while $H_\F(\{a_m\})=\{0,1\}$ and $H_\F(\{0\})=\{0\}$; thus $H_\F$ is not continuous at $\{0\}$, and hence not pointwise equicontinuous. Note that $H_\F^*\equiv\{0\}$ is constant and (T) holds; Theorem~\ref{thm:cont-eqq} does not apply since $H_\F$ is not generated by a continuous multifunction ($F(x)=H_\F(\{x\})$ is discontinuous at $0$).
}}
\end{example}

%{Moreover, let us observe that the assumption on inner attractivity is essential.
%\begin{example}
 %   \emph{
%Let $S$ be the unit circle in the complex plane $\mathbb{C}$, let $f_1,f_2:S\to S$ be defined by
%$$
%f_1\big(e^{2\pi i t}\big):=e^{2\pi i \sqrt{t}}\;\;\;f_2(z):=z
%$$
%and consider the IFS $\F=\{f_1,f_2\}$.
%Then the Hutchinson operator $H_\F$ is weakly Picard, non inner attracting, and $H_\F$ is not pointwise equicontinuous.\\
%To see that $H_\F$ is weakly Picard, observe that for every $K\in\K(S)$, the sequence $(H_\F^{(n)}(K))$ is an increasing sequence of compact sets, so
%$$
%H_\F^{(n)}(K)\to \overline{\bigcup_{k\in\N}H_\F^{(n)}(K)}
%$$
%$H_\F$ is not inner attracting as
%$$
%H_\F^{(n)}(S)\to S
%\;\;\;\;\;\mbox{
%and}\;\;\;\;\;
%H^{(n)}_\F(\{1\})=\{1\}\to\{1\}
%$$
%Finally, $H_\F$ is not equicontinuous as (cf. Example \ref{example:com})
%$$
%H^{(n)}_\F(\{1\})=\{1\}
%$$
%and
%$$
%-1\in H^{(n)}_\F\big(\big\{e^{2\pi i \big(\frac{1}%{2}\big)^{2^n}}\big\}\big)
%$$
%    }
%\end{example}}
{
Finally, let us observe that condition (T) is not necessary for the assertion of Theorem \ref{thm:cont-eqq}.%\footnote{{\color{red}\sout{In original version of the paper there was no such an example}}}
\begin{example}\label{ex:idzero}\emph{
Let $H_\F$ be the Hutchinson operator generated by the IFS $\F=\{f_1,f_2\}$ on the unit interval $X=[0,1]$, where $f_1(x):=x$ and $f_2(x):=0$ for $x\in[0,1]$. Then for every $K\in\K([0,1])$, we have $H_\F(K)=K\cup\{0\}$, and hence also $H_\F^*(K)=K\cup\{0\}$. Thus $H_{\F}$ is weakly Picard, $H_\F^*$ is continuous and $H_\F$ is nonexpansive, hence uniformly equicontinuous. However, (T) fails: {if $x\in(0,1]$, then $0\in\{0,x\}=H_\F^*(\{x\})$, while $H_\F^*(\{0\})=\{0\}\neq\{0,x\}$; by Lemma~\ref{lem:T}, $H_\F$ does not satisfy condition (T)}.}
\end{example}
}

The following result is a particular version of Theorem \ref{thm:cont-eqq}. 
Since we use it later in the proof of the general case, 
we state it here as a lemma. 
(This lemma can be also deduced from the main result of~\cite{Sarizadeh2}
using Lemma~\ref{lem:sing}.)

\begin{lemma}\label{le:X-eqq}
Let $X$ be a compact metric space and $H:\K(X)\to\K(X)$ be a Hutchinson operator.
If $X$ is a contracting fixed point of $H$, then $H$ is uniformly equicontinuous.
\end{lemma}

\begin{proof}
By Lemmas \ref{lem:wPic1} and \ref{lem:sing}, it is enough to show that
\begin{equation}\label{eq:fi:eksperyment}
\forall_{\varepsilon>0}\;\exists_{n_0\in\N}\;\forall_{n\geq n_0}\;\forall_{x\in X}\;h(H^{(n)}(x),X)\leq \varepsilon.
\end{equation}
Indeed, by Lemma \ref{lem:wPic1}, this condition implies uniform equicontinuity of $H$ on $X$ (here we again identify $X$ with $\{\{x\}\colon x\in X\}$), 
which gives its uniform equicontinuity on the whole $\K(X)$ 
by~Lemma \ref{lem:sing}.

Fix any $\varepsilon>0$. For any $k\in\N$, define
$$
X_{k}=\{x\in X:\forall_{n\geq k}\;h(H^{(n)}(x),X)\leq \varepsilon\}.
$$
%where $$\overline{A}(\varepsilon)=\bigcup_{x\in A}\overline{B}(x,\delta)=\{x\in X:\;\exists_{y\in A}\;d(x,y)\leq{\varepsilon}\}.$$
Each set $X_{k}$ is closed as its complement
$$
\{x\in X:\exists_{n\geq k}\;h(H^{(n)}(x),X)>\varepsilon\}
$$
is open by continuity of $H^{(n)}$.\\
% Indeed, assume that $(x_n)\subset X_{k}$ and $x_n\to x$, and take $m\geq k$. By continuity of $H^{(m)}$, we have $h(H^{(m)}(x_n),H^{(m)}(x))\to 0$. By continuity of a metric, we have
%$$
%h(H^{(m)}(x),X)=\lim_{n\to\infty}h(H^{(m)}(x_n),X)\leq\varepsilon,
%$$
%and hence $x\in X_m$.
Now since $X$ is a contracting fixed point of $H$, we have that $X=\bigcup_{m\in\N}X_m$. 
Since $X$ is complete, by the Baire Category theorem, 
one of sets $X_{k_0}$ is not nowhere dense in $X$. 
Since $X_{k_0}$ is closed, there is an open set $\emptyset \neq U\subset X_{k_0}$. Now for any $n\in\N$, let
$$
B_n:=\{x\in X\colon H^{(n)}(x)\cap U\neq\emptyset\}.
$$
Again by the fact that $X$ is a contracting fixed point of $H$, we have that $X=\bigcup_{n\in\N} B_n$, 
and since each $H^{(n)}$ is continuous, we also have that each $B_n$ is open. 
Since $X$ is compact, we can find 
$n_1,\ldots,n_j$ 
such that $X=\bigcup_{i=1}^j B_{n_i}$.

Fix $n_0:=\max\{n_1,\ldots,n_j\}+k_0$ and choose any $n\geq n_0$ and $x\in X$. 
Then $x\in B_{n_i}$ for some $i\in\{1,\ldots,j\}$, 
meaning that $H^{(n_i)}(x)\cap U\neq\emptyset$. 
Take $x'\in H^{(n_i)}(x)\cap U$. In particular, $x'\in X_{k_0}$. 
Since $n-n_i\geq k_0$, we have 
$h(H^{(n-n_i)}(x'),X)\leq\varepsilon$. 
Hence, by the fact that
$$
H^{(n-n_i)}(x')\subset H^{(n-n_i)}(H^{(n_i)}(x))=H^{(n)}(x)\subset X,
$$
we have $h(H^{(n)}(x),X)\leq\varepsilon$ and (\ref{eq:fi:eksperyment}) holds.
\end{proof}

%$A\in\K(X)$ is its \emph{inner attracting fixed point}, if $A\subset X_H$ and $H^*(x)=A$ for every $x\in A$.\\
%We say that $H$ is \emph{proper}, if $X\subset X_H$ and $H^*(x)$ is inner attracting for every $x\in X$.
\begin{lemma}\label{le:V}
Let $H:\K(X)\to\K(X)$ be a weakly Picard Hutchinson operator satisfying condition (T).
Then for every $x_0\in X$ and every $\varepsilon>0$,
there exists $n_0\in\N$ and an open set $V\supset H^*(\{x_0\})$ with
\begin{equation}
h(H^{(n_0)}(y),H^*(\{x_0\}))<\varepsilon\textup{\ for each\ }y\in V.
\end{equation}
\end{lemma}

\begin{proof}
Choose $x_0\in X$ and set $A:=H^*(\{x_0\})$.
Let $H_{A}$ be the restriction of $H$ to the subspace $\K(A)$. Then $H_A$ is the Hutchinson operator generated by the (necessarily continuous) multivalued map $F_{\restriction A}:A\to\K(A)$. Indeed, since $H(A)=A$, we have that $F(x)\subset A$ for any $x\in A$, and for $D\in\K(A)$,
$$
H_A(D)=H(D)=\bigcup_{x\in D}F(x)=\bigcup_{x\in D}F_{\restriction A}(x).
$$
Moreover, $A$ is a contracting fixed point of $H_A$ since for every $D\in\K(A)$, we have
$$
h(H_A^{(n)}(D),A)=h(H^{(n)}(D),A)\to 0,
$$
because condition (T), applied to $x_0$ and to the nonempty compact set $D\subset A=H^*(\{x_0\})$, gives $H^*(D)=H^*(\{x_0\})=A$, i.e.\ $H^{(n)}(D)\to A$.
Now we can use Lemma \ref{le:X-eqq} and get that $H_A$ is uniformly equicontinuous, and then Lemma  \ref{lem:wPic1} for $H_A$ (and $X=Y=\K(A)$) to get
\begin{equation}\label{eq:fi2:eksperyment}
\forall_{\varepsilon>0}\;\exists_{n_0\in\N}\;\forall_{x\in A}\;h(H^{(n_0)}(x),A)< \frac{\varepsilon}{2}.
\end{equation}
Take any $\varepsilon>0$ and fix $n_0$ as in (\ref{eq:fi2:eksperyment}). 
Since $H^{(n_0)}$ is continuous, we have
\begin{equation}\label{eq:fi3}
\forall_{x\in A}\;\exists_{r_x>0}\;\forall_{y\in B(x,r_x)}\;h(H^{(n_0)}(y),H^{(n_0)}(x))<\frac{\varepsilon}{2}.
\end{equation}
Let $V:=\bigcup_{x\in A}B(x,r_x)$ be an open subset of $X$ containing $A$. Take any $y\in V$. Then $y\in B(x,r_x)$ for some $x\in A$ and hence
\begin{equation}\label{eq:fi4}
h(H^{(n_0)}(y),A)\leq 
h(H^{(n_0)}(y),H^{(n_0)}(x))+h(H^{(n_0)}(x),A)<
\frac{\varepsilon}{2}+\frac{\varepsilon}{2}=\varepsilon.
\end{equation}
%\begin{equation}\label{eq:fi3:eksperyment}
%\forall_{x\in A}\;\forall_{y\in B(x,r)}\;h(H^{(n)}(y),H^{(n)}(x))<\frac{\varepsilon}{2}.
%\end{equation}
%Set $V_n:=\bigcup_{x\in A}B(x,r)$ and take any $y\in V_n$. Then $y\in B(x,r)$ for some $x\in A$ and thus 
%\begin{equation}\label{eq:fi4:eksperyment}
%h(H^{(n)}(y),A)\leq 
%h(H^{(n)}(y),H^{(n)}(x))+h(H^{(n)}(x),A)<\frac{\varepsilon}{2}+\frac{\varepsilon}{2}=\varepsilon.
%\end{equation}
\end{proof}

\begin{proof}[Proof of Theorem~\ref{thm:cont-eqq}]
The implication ``(\ref{thm:cont-eqq:2})$\Rightarrow$(\ref{thm:cont-eqq:1})''
follows from Lemma \ref{lem:eqcont1}. We will prove the opposite one.

Assume that $H^*:\K(X)\to \K(X)$ is continuous. 
It is enough to show that $H$ is uniformly equicontinuous on each compact set $Y\subset X$ . 
Indeed,  by Lemma \ref{lem:sing}, $H$ is then uniformly continuous on each $\K(Y)$ for compact $Y\subset X$.
Now if $E\subset \K(X)$ is compact, then $Y:=\bigcup\{K:K\in E\}$ is compact and $E\subset \K(Y)$. Hence uniform equicontinuity on $\K(Y)$ implies uniform equicontinuity on $E$, so by Lemma \ref{lem:newkey}, pointwise equicontinuity of $H$.

%assume that it is the case. Then which implies pointwise equicontinuity of $H$:
%$$
%\forall_{x_0\in X}\;\forall_{r>0}\;\forall_{\varepsilon>0}\;\exists_{\delta>0}\;\forall_{n\in\N}\;\forall_{K,D\in\K( B(x_0,r))}\;h(K,D)<\delta\;\Rightarrow\;h(H^{(n)}(K),H^{(n)}(D))<\varepsilon
%$$ 
%which in particular gives
%\begin{equation}\label{eq:final11}
%\forall_{K_0\in\K(X)}\;\forall_{\varepsilon>0}\;\exists_{\delta>0}\;\forall_{n\in\N}\;\forall_{D\in B(K_0,\delta)}\;h(H^{(n)}(K_0),H^{(n)}%(D))<\varepsilon
%\end{equation}
%Indeed, take $K_0\in\K(X)$ and take any $r>0$. Then $Y:=\overline{\bigcup_{x\in K_0}B(x,r)}$ is closed and bounded, hence compact. Now let $\varepsilon>0$. By uniform equicontinuity of $H$ on $\K(Y)$, we can find positive $\delta<r$ such that for any $n\in\N$ and $K,D\in B(K_0,\delta)\cap Y$ we have $h(H^{(n)}(K),H^{(n)}(D))<\varepsilon$. If $D\in B(K_0,\delta)$, then $K_0,D\in B(K_0,\delta)\cap Y$, so  $h(H^{(n)}(K_0),H^{(n)}(D))<\varepsilon$ for all $n\in\N$ and we get (\ref{eq:final11}).\\

So our aim is to prove that $H$ is uniformly equicontinuous on each compact set $Y\subset X$. 
By Lemma \ref{lem:wPic1} (with identifying $Y$ with $\{\{x\}\colon x\in Y\}$), we have to prove the following:
\begin{equation}\label{eq:final12}
\forall_{\varepsilon>0}\;\exists_{n_1\in \N}\;\forall_{n\geq n_1}\;\forall_{x\in Y}\;h(H^{(n)}(x),H^*(x))< \varepsilon.
\end{equation}
Fix $\varepsilon>0$. We claim that for every $x_0\in Y$ there exist $r_{x_0}>0$ and $n_{x_0}\in\N$ with
\begin{equation}\label{eq:cover}
h(H^{(n)}(x),H^*(\{x\}))<\varepsilon\quad\textup{for all }x\in Y\cap \overline{B}(x_0,r_{x_0})\textup{ and }n\geq n_{x_0}.
\end{equation}
Then, since $Y$ is compact there are finitely many $x_0^1,\ldots,x_0^p$ with $Y\subset\bigcup_{j=1}^p B(x_0^j,r_{x_0^j})$, 
and then $n_1:=\max\{n_{x_0^1},...,n_{x_0^p}\}$ satisfies (\ref{eq:final12}).

To prove the claim, fix $x_0\in Y$ and apply Lemma \ref{le:V} to the point $x_0$ with $\tfrac{\varepsilon}{2}$: there are $n_0\in\N$ and an open $V\supset H^*(\{x_0\})$ with
$$
h(H^{(n_0)}(y),H^*(\{x_0\}))<\tfrac{\varepsilon}{2}\quad\textup{for every }y\in V.
$$
For $r>0$ set $Y_r:=Y\cap\overline{B}(x_0,r)$; it is nonempty, compact, and $h(Y_r,\{x_0\})\leq r$. 
Since $H^*$ is continuous and $H^*(\{x_0\})\subset V$ with $V$ open, we may find $r_{x_0}>0$ 
such that
$$
H^*(Y_{r_{x_0}})\subset V\qquad\textup{and}\qquad h(H^*(\{y\}),H^*(\{x_0\}))<\tfrac{\varepsilon}{2}\ \textup{ for }y\in Y_{r_{x_0}}.
$$
As $H^{(n)}(Y_{r_{x_0}})\to H^*(Y_{r_{x_0}})\subset V$ and $V$ is open, there is $n'\in\N$ with $H^{(n)}(Y_{r_{x_0}})\subset V$ for all $n\geq n'$; put $n_{x_0}:=n_0+n'$.

Finally, take $x\in Y_{r_{x_0}}$ and $n\geq n_{x_0}$. Then $H^{(n-n_0)}(x)\subset H^{(n-n_0)}(Y_{r_{x_0}})\subset V$, so
$$
h(H^{(n)}(x),H^*(\{x_0\}))=h\big(H^{(n_0)}(H^{(n-n_0)}(x)),H^*(\{x_0\})\big)\leq\sup_{y\in V}h(H^{(n_0)}(y),H^*(\{x_0\}))\leq\tfrac{\varepsilon}{2},
$$
and therefore, since $x\in Y_{r_{x_0}}$,
$$
h(H^{(n)}(x),H^*(\{x\}))\leq h(H^{(n)}(x),H^*(\{x_0\}))+h(H^*(\{x_0\}),H^*(\{x\}))<\tfrac{\varepsilon}{2}+\tfrac{\varepsilon}{2}=\varepsilon.
$$
This proves (\ref{eq:cover}), and finishes the proof.
\end{proof}
%{\color{red}\sout{\textbf{Acknowledgments}}\\
%\sout{Once more we would like to express our thanks the anonymous reviewer who pointed out that Hutchinson operators with attractors need not be continues and provided us Example \ref{ex:referee}.}}

%%%%%%%%%%%%%%%%%%%%%%%%%%%%%%%%%%%%%%%%%%%%%%%%%%%%%%%%%%%%%%%%%%%%%%%%%%%%%%%%
% Bibliography
%%%%%%%%%%%%%%%%%%%%%%%%%%%%%%%%%%%%%%%%%%%%%%%%%%%%%%%%%%%%%%%%%%%%%%%%%%%%%%%%

\bibliographystyle{alpha}
\bibliography{equiIFS}

\end{document}